\documentclass[11pt, oneside]{article}   	              		
\usepackage{graphicx}
\usepackage[margin=3cm]{geometry}						
\usepackage{amssymb, amsmath, amsthm, mathtools, xcolor, stmaryrd, hyperref, enumitem}

\title{Note on a certain category of mod $p$ representations}
\author{Reinier Sorgdrager}
\date{}
\DeclareMathOperator{\GL}{GL}
\DeclareMathOperator{\Gal}{Gal}

\newcommand{\gres}{\text{gr}_{N,\text{res}}}
\newcommand{\grint}{\text{gr}_{N,\text{int}}}

\DeclareMathOperator{\gr}{gr}

\newtheorem{proposition}{Proposition}

\newtheorem{lemma}{Lemma}
\newtheorem{theorem}{Theorem}
\newtheorem{warning}{Warning}

\theoremstyle{definition}
\newtheorem{definition}{Definition}
\begin{document}
	\maketitle
	\abstract{Let $p>3$ be a prime number, $f\geq1$ an integer. We consider a certain full subcategory $\mathcal C$ of the category of smooth admissible mod $p$ representations of either $\GL_2\mathbf Q_{p^f}$ or of the group of units of the quaternion algebra over $\mathbf Q_{p^f}$. This category was introduced in the context of the mod $p$ Langlands program by \cite{BHHMSConjs} in the $\GL_2$-case and by \cite{hu2024modprepresentationsquaternion} in the quaternion case.\\
		We prove that whether a smooth admissible mod $p$ representation $\pi$ (with central character) belongs to $\mathcal C$ is completely determined by the restriction of $\pi$ to an arbitrarily small open subgroup.}
	\section{Introduction}
	Let $f\geq1$ and $K=\mathbf Q_{p^f}$, for $p$ a prime bigger than $3$. Let $\mathbf F$ be a finite field containing $\mathbf F_{p^f}$.\\
	We consider a $p$-adic Lie group $G$ which is either
	\begin{itemize}
		\item\textbf{$\GL_2$-case}: the group $I_1/Z_1$, where $I_1=\begin{pmatrix}
			1+p\mathcal O_K&\mathcal O_K\\p\mathcal O_K&1+p\mathcal O_K
		\end{pmatrix}\subset\GL_2K$ is the upper-triangular pro-$p$ Iwahori subgroup and $Z_1$, the center, its subgroup of scalar matrices,\end{itemize} or \begin{itemize}
		\item \textbf{quaternion case}: the group $U_1/Z_1$, where $U_1=1+\Pi\mathcal O_D\subset D^\times$ is the subgroup of the units of the quaternion algebra $D$ over $K$ of invariant $1/2$ of elements that are $1$ modulo a uniformizer $\Pi$ (that we choose to be a square root of $p$); the group $Z_1=1+p\mathcal O_K$ then denotes the center.
	\end{itemize}
	In both cases the $p$-adic Lie group $G$ is pro-$p$ of dimension $3f$.\\
	
	In \cite{BHHMSConjs} and \cite{hu2024modprepresentationsquaternion} a certain category $\mathcal C$ of mod $p$ representations $\pi$ of the bigger group $G_\text{big}$ ($\GL_2K$ or $D^\times$ depending on which case we are in) is introduced whose definition is given in terms of certain invariants that can be attached to $\pi|_G$\footnote{The restriction to $G$ will make sense because the center by which one mods out to obtain $G$ will always act trivially on these $\pi$ (as they will be required to have a central character).}. This category is of interest in the mod $p$ Langlands program.\\ 
	We remind the reader that people so far have not been able to construct a mod $p$ Langlands further than for $\GL_2\mathbf Q_p$. A significant obstacle is that as soon as $f>1$, the smooth admissible mod $p$ representations of $\GL_2K$ become much more numerous than the representations on the Galois-side. One of the challenges is therefore to identify what representations of $\GL_2K$ can be in the image of a conjectured mod $p$ Langlands correspondence. The category $\mathcal C$ provides a step in this direction.\\ 
	Indeed, in the $\GL_2$-case the category $\mathcal C$ is known to contain the candidates $\pi_{\overline\rho}$ for a conjectured mod $p$ Langlands correspondence $\overline\rho\mapsto\pi_{\overline\rho}$ from mod $p$ $\Gal(\overline{\mathbf Q}_p/K)$-representations to mod $p$ $\GL_2K$-representations: inspired by Emerton's local-global compatibility one can construct such candidates by global methods, via the mod $p$ cohomology of Shimura varieties, and provided $\overline\rho$ satisfies sufficient genericity assumptions it is proven in \cite{BHHMS1} that the resulting mod $p$ representation of $\GL_2K$ is in $\mathcal C$ (see the proof of Thm.~5.3.5 in \emph{loc.~cit.} and their Thm.~6.4.7, see also \cite[Thm.~5.1]{wang2023modpcohomologyoperatornamegl2}).\\
	In the $D^\times$-case the category $\mathcal C$ was considered in the paper \cite{hu2024modprepresentationsquaternion}, a work motivated by the search for a mod $p$ Jacquet-Langlands correspondence. In the case $f=1$ the authors similarly prove that mod $p$ representations of $D^\times$ coming from a sufficiently generic mod $p$ representation of $\Gal(\overline{\mathbf Q}_p/\mathbf Q_p)$, via a construction using the cohomology of a suitable Shimura variety, is in $\mathcal C$ (the quaternion version this time), see Thm.~6.14 of \emph{loc.~cit.}.\\
	The category $\mathcal C$ is still mysterious and far from well-understood. Let us remark for example that $\pi$ being in $\mathcal C$ implies that its \emph{Gelfand-Kirillov dimension} is $\leq f$ (see the proof of \cite[Thm.~5.3.5]{BHHMS1}), but that it is unknown whether the converse holds.\\
	
	The main purpose of this note is to shed some more light on the category $\mathcal C$ by proving a property that might be useful in its further study. Namely, we prove that whether a $\pi$ is in $\mathcal C$ can be detected on the restriction of $\pi$ to an arbitrarily small open subgroup. More precisely, we prove the following theorem:
	\begin{theorem}\label{main}
		Let $\pi$ and $\pi'$ be smooth representations of $G_\text{big}$ ($\GL_2K$ or $D^\times$) with same fixed central character. Suppose $\pi'\in\mathcal C$ and $\pi|_H\cong\pi'|_H$ for some open subgroup $H$ of $G_\text{big}$. Then $\pi\in\mathcal C$.
	\end{theorem}
	\begin{warning}
		We warn the reader that what we call the category $\mathcal C$ in the $\GL_2$-case is actually not known to be the same as the category denoted by ``\hspace{0.25em}$\mathcal C$'' in \cite{BHHMSConjs}. We only know it is a full subcategory thereof, which is explained by the paragraph at the top of page 115 of \emph{loc.~cit.}. However, it is suspected in \cite[Rem.~3.3.2.6(ii)]{BHHMSConjs} that the two categories are the same.
	\end{warning}
	We now give some details on the definition of the category $\mathcal C$. In both cases the category $\mathcal C$ is a full subcategory of the category of smooth admissible representations of $G_\text{big}$ over $\mathbf F$ that have a fixed central character (which we abusively omit from the notation).\\
	An object $\pi$ of this ambient category has $Z_1$ acting trivially on it, because $Z_1$ is pro-$p$ and central and acting via a character to $\mathbf F^\times$, which is therefore necessarily trivial. This means $\pi$ is also a representation of $G$. Its dual $\pi^\vee$ is then a pseudocompact $\mathbf F\llbracket G\rrbracket$-module, where $\mathbf F\llbracket G\rrbracket$ is the completed group ring, i.e. the limit of all group rings over $\mathbf F$ of quotients of $G$ by normal open subgroups.\\
	Since $G$ is pro-$p$, the completed group ring is local, having a unique (two-sided) maximal ideal $\mathfrak m_G$. One can form the associated graded ring $\gr\mathbf F\llbracket G\rrbracket\coloneqq\bigoplus_{i\geq0}\mathfrak m_G^i/\mathfrak m_G^{i+1}$ which acts on $\gr\pi^\vee\coloneqq\bigoplus_{i\geq0}\mathfrak m_G^i\pi^\vee/\mathfrak m_G^{i+1}\pi^\vee$, making the latter into a graded module.\\
	Next, a certain two-sided ideal $J_{\mathcal C}\subset\gr\mathbf F\llbracket G\rrbracket$ is identified: in the $\GL_2$-case it is the ideal $I_G$ defined just before \cite[Thm.~5.3.4]{BHHMS1} and in the quaternion case it is the ideal $I_D$ defined just before \cite[Cor.~2.10]{hu2024modprepresentationsquaternion}.\\
	The definition of $\mathcal C$ is now straightforward: its objects are all $\pi$ for which $\gr\pi^\vee$ is killed by a power of $J_{\mathcal C}$. It turns out to be an Abelian subcategory of the smooth admissible representations with the fixed central character and it is stable under subquotients and extensions.\\
	
	Let us briefly say a few words about our strategy. It is roughly to associate to any $\pi|_H$ some graded module ($\gres\pi|_H^\vee$ in notation that will make sense later) over some graded subalgebra of $\gr\mathbf F\llbracket G\rrbracket$ and to show that $\pi\in\mathcal C$ is equivalent to this graded module being killed by some power of a certain ideal.\\
	We will see that we can define the $\mathfrak m_G$-adic filtration on $\pi^\vee$ giving rise to $\gr\pi^\vee$ in a language that is more amenable to restricting to subgroups. Namely, we can define the $\mathfrak m_G$-adic filtration in terms of $p$-valuations and ordered bases. The relevant concepts for this will be introduced in the next two sections. Afterwards, we define the new association of graded modules $\gres(-)$. In order to prove Theorem \ref{main} we then establish a series of lemmas that phrase the condition of being in $\mathcal C$ in terms of the new graded modules and conclude.\\
	We note that we do not need to know the precise description of the ideal $J_{\mathcal C}$. All that matters is that $\gr\mathbf F\llbracket G\rrbracket/J_{\mathcal C}$ is commutative. Therefore our Theorem \ref{main} generalizes to any category $\mathcal C'$ defined in the same way but with respect to a different such ideal $J_{\mathcal C'}$.
		\subsection*{Acknowledgments}
	The author is grateful to Florian Herzig and Benjamin Schraen as well as to an anonymous referee for their comments and corrections on earlier versions of this article.
	\section{$p$-valuations}
	We recall the definition of a $p$-valuation.
	\begin{definition}
		Let $H$ be a group. A \emph{$p$-valuation} on $H$ is a map $$\omega\colon H\to(\tfrac1{p-1},\infty]$$ such that for all $x,y\in H$\begin{itemize}
			\item$\omega(xy^{-1})\geq\min\{\omega(x),\omega(y)\}$,
			\item$\omega([x,y])\geq\omega(x)+\omega(y)$,
			\item$\omega(x)=\infty$ if and only if $x$ is trivial,
			\item$\omega(x^p)=\omega(x)+1$.
		\end{itemize}
	\end{definition}
	One can use a $p$-valuation to define a topology on $H$, using the inverse images of opens around $\infty$ to get a system of open neighbourhoods around the identity element of $H$.\\
	In our case, it turns out there is a $p$-valuation $\omega$ on $G$ which induces its topology. Moreover, this $p$-valuation is \emph{saturated}, which means that for all $x$, $\omega(x)>\frac p{p-1}$ implies $\exists y\in G\colon y^p=x$.\\
	One of the important consequences of the saturatedness of $G$ is that the sets $$G^{p^N}\coloneqq\{g^{p^N}\mid g\in G\},$$ are open subgroups of $G$ for all $N\geq0$. Since $\omega$ defines the topology of $G$ the groups $G^{p^N}$ get arbitrarily small as $N\to\infty$. For our purposes it therefore suffices to work with the $G^{p^N}$.\\
	
	In the $\GL_2$-case, see \cite[\S5.3]{BHHMS1} for more details and in particular the construction of $\omega$. For more general details on $p$-valuations, see \cite[Ch.~V]{Schneider2011}.\\ We give some more details on $\omega$ in the quaternion case, as this is not in \cite{hu2024modprepresentationsquaternion}. This $p$-valuation is constructed via the $p$-adic valuation $v_p$ on $D$, normalized so that $v_p(p)=1$, and via the reduced norm (actually we do it even more down-to-earth). Namely, one can consider the map from $D$ to $(-\infty,\infty]$ sending $x$ to $v_p(x-1)$. It seems a good candidate, were it not that $G$ is not (evidently) a subgroup of $D^\times$ and we cannot simply restrict.\\
	Fortunately, this can be resolved by realizing that $G$ naturally is a subgroup of $D^\times$: indeed, after choosing a basis, we get a $K$-algebra morphism $D\to\text{Mat}_{4\times4}K$; from the target we have the determinant morphism, and the image of $U_1=1+\Pi\mathcal O_D$ lands in $Z_1=1+p\mathcal O_K$ on which we have the morphism $\sqrt[4]-$; lastly, the composition gives us a retract $U_1\to Z_1$ of the inclusion of the center allowing us to write $G\times Z_1\cong U_1\subset D^\times$.\\
	Then we can continue our strategy, and define $\omega$ on $G$ by $\omega(x)\coloneqq v_p(x-1)$ via the identification as a subgroup of $D^\times$. One checks that this indeed defines a $p$-valuation, using $p>3$. Since the $p$-powers in $1+\Pi\mathcal O_D$ are precisely $1+p\Pi\mathcal O_D$, it follows that $\omega$ is saturated. Also, it is immediate from the definition that the topology induced by $\omega$ is the one of $G$.
	\section{Ordered basis}
	A nice property of $G$ is that it has an \emph{ordered basis}. That is, there are $g_1,\dots,g_n\in G$ such that the map $$\mathbf Z_p^n\to G,(x_1,\dots,x_n)\mapsto g_1^{x_1}\cdot\cdots\cdot g_n^{x_n}$$ is a homeomorphism. Moreover, one has $\omega(g_1^{x_1}\cdot\cdots\cdot g_n^{x_n})=\min_i\omega(g_i)+v_p(x_i)$. The number $n$, which is called the \emph{rank} of $G$, equals its dimension and is therefore $3f$. See \cite[\S26]{Schneider2011} for more information about ordered bases.\\
	We obtain an ordered basis by letting $(g_1,\dots,g_{3f})=(A_0,\dots,A_{f-1},B_0,\dots,B_{f-1},C_0,\dots,C_{f-1})$, where \begin{align*}
		&A_i=\begin{cases}
			\begin{pmatrix}
				1&[\alpha^i]\\0&1
			\end{pmatrix}&\text{in the $GL_2$-case,}\\
			1+[\alpha^i]\Pi&\text{in the quaternion case,}
		\end{cases}\\
		&B_i=\begin{cases}
			\begin{pmatrix}
				1&0\\{[\alpha^i]}p&1
			\end{pmatrix}&\text{in the $GL_2$-case,}\\
			1+[\alpha^i][\zeta]\Pi&\text{in the quaternion case,}
		\end{cases}\\
		&C_i=\begin{cases}
			\begin{pmatrix}
				1+[\alpha^i]p&0\\0&(1+[\alpha^i]p)^{-1}
			\end{pmatrix}&\text{in the $GL_2$-case,}\\
			1+[\alpha^i][\zeta]p&\text{in the quaternion case,}
		\end{cases}
	\end{align*}
	here $[-]$ denotes the Teichm\"uller lift, $\alpha$ is a generator of $\mathbf F_{p^f}$ over $\mathbf F_p$, and $\zeta$ generates the quadratic extension of $\mathbf F_{p^f}$ over $\mathbf F_{p^f}$. In the $\GL_2$-case see the proof of \cite[Prop.~5.3.3]{BHHMS1} for a proof; in the quaternion case combine \cite[Prop.~2.1]{hu2024modprepresentationsquaternion} (but mod out by the center) with \cite[Prop.~26.5]{Schneider2011}.\\ We note that $\omega(A_i)=\omega(B_i)=\frac12$ and $\omega(C_i)=1$ for all $0\leq i\leq f-1$.\\
	
	The ordered basis gives us the following description of the completed group ring: \begin{equation}\label{complgrpring}
		\mathbf F\llbracket G\rrbracket=\left\{\sum_{\underline x\geq0}c_{\underline x}z^{\underline x}\colon c_{\underline x}\in\mathbf F,\text{ where }z^{\underline x}\coloneqq(g_1-1)^{x_1}\cdot\cdots\cdot(g_{3f}-1)^{x_{3f}}\right\}.
	\end{equation} Note that elements have unique such power series expansions, because $\mathbf F\llbracket G\rrbracket$ is isomorphic to $\mathbf F\llbracket\mathbf Z_p^{3f}\rrbracket\cong\mathbf F\llbracket X_1,\dots,X_{3f}\rrbracket$ as a topological module via the homeomorphism $\mathbf Z_p^{3f}\to G$ given by the ordered basis.\\
	We define a map $\nu\colon\mathbf F\llbracket G\rrbracket\to\mathbf Z_{\geq0}\cup\{\infty\}$ by $\nu\coloneqq2\overline w$ where $$\overline w\left(\sum_{\underline x\geq0}c_{\underline x}z^{\underline x}\right)\coloneqq\inf\left\{\sum_{i=1}^{3f}x_i\omega(g_i)\colon c_{\underline x}\neq0\right\}.$$ One can check that it is a valuation. It defines a filtration none other than the $\mathfrak m_G$-adic one:
	\begin{proposition}\label{maxideals}
		For $j\in\mathbf Z_{\geq0}$ we have $$\mathfrak m_G^j=\{x\in\mathbf F\llbracket G\rrbracket\colon\nu(x)\geq j\}.$$
	\end{proposition}
	\begin{proof}
		In the $\GL_2$-case this is precisely \cite[Prop.~5.3.3]{BHHMS1}. The proof in the quaternion case is very similar: the $\subset$-part is proven by remarking that the case for general $j$ follows from the case for $j=1$ (which is obvious) together with $\nu(xy)=\nu(x)+\nu(y)$; for the $\supset$-part it suffices by (\ref{complgrpring}) to prove $C_i-1\in\mathfrak m_G^2$ and for this it suffices to show $C_i\in[G,G]G^p$ as in \emph{loc.~cit.}. Indeed, modulo $\mathfrak m_G^2$ one has $(gh-1)=g-1+h-1$ for $g,h\in G$ and therefore any $g$ in $[G,G]G^p$ satisfies $g-1\in\mathfrak m_G^2$ (to deal with the $p$-powers, simply notice that $g^p-1=(g-1)^p\in\mathfrak m_G^2$). To show $C_i\in[G,G]G^p$ now, remark first that for $\gamma\in\mathcal O_K$ one has $$[(1+[\zeta]\Pi),(1+\gamma\Pi)]=1+\gamma([\zeta]-[\zeta^{p^f}])p\mod p\Pi\mathcal O_D,$$ which means precisely that both sides of the equation are the same up to $p$-powers in $1+\Pi\mathcal O_D$.\\
		It suffices to find a $\gamma\in\mathcal O_K$ for which the class of $1+\gamma([\zeta]-[\zeta^{p^f}])p$ in $(1+p\mathcal O_D)/\left((1+p\Pi\mathcal O_D)(1+p\mathcal O_K)\right)$ is the same as that of $C_i$. Since $(1+p\mathcal O_D)/(1+p\Pi\mathcal O_D)\cong k[\zeta]$ and since $\zeta-\zeta^{p^f}\not\in k$ there is indeed such a $\gamma$, as $(1+p\mathcal O_D)/\left((1+p\Pi\mathcal O_D)(1+p\mathcal O_K)\right)\cong k[\zeta]/k$ is $1$-dimensional over $k$.
	\end{proof}
	We also note that the choice of ordered basis for $G$ gives us canonical ordered bases for the groups $G^{p^N}$, $N\geq1$:
	\begin{proposition}
		$G^{p^N}$ has an ordered basis given by $\left(g_1^{p^N},\dots,g_{3f}^{p^N}\right)$.
	\end{proposition}
	\begin{proof}
		This follows from $G^{p^N}$ being the image of $p^N\mathbf Z_p^{3f}$ under the homeomorphism $\mathbf Z_p^{3f}\to G$, which is an immediate consequence of the saturatedness of $\omega$ and the formula $\omega(g_1^{x_1}\cdot\cdots\cdot g_n^{x_n})=\min_i\omega(g_i)+v_p(x_i)$.
	\end{proof}
	Consequently \begin{equation}\label{groupringpN}\mathbf F\llbracket G^{p^N}\rrbracket=\left\{\sum_{\underline x\geq0}c_{\underline x}z^{\underline x}\colon \underline x\in p^N\mathbf Z^{3f}\right\}\subset\mathbf F\llbracket G\rrbracket,\end{equation} using the same notation as in (\ref{complgrpring}).
	\section{Associated graded rings and modules}
	From now on we fix an $N\geq1$.\\
	
	Restricting $\nu$ to $\mathbf F\llbracket G^{p^N}\rrbracket$ induces a filtration with associated graded algebra $$\gr_\nu\mathbf F\llbracket G^{p^N}\rrbracket\coloneqq\bigoplus_{i\geq0}\mathbf F\llbracket G^{p^N}\rrbracket_{\nu\geq i}/\mathbf F\llbracket G^{p^N}\rrbracket_{\nu>i}.$$
	By Proposition \ref{maxideals} it is a graded subalgebra of $\gr\mathbf F\llbracket G\rrbracket$. By (\ref{groupringpN}) it is actually $p^N\mathbf Z$-graded. That is, the quotient $$\mathbf F\llbracket G^{p^N}\rrbracket_{\nu\geq i}/\mathbf F\llbracket G^{p^N}\rrbracket_{\nu>i}$$ is only non-zero for $i$ of the form $kp^N$ and, in fact, then  one has $$\mathbf F\llbracket G^{p^N}\rrbracket_{\nu>kp^N}=\mathbf F\llbracket G^{p^N}\rrbracket_{\nu\geq(k+1)p^N}.$$
	If $M$ is a $\mathbf F\llbracket G\rrbracket$-module, we can now associate to it the following three graded $\gr_\nu\mathbf F\llbracket G^{p^N}\rrbracket$-modules:
	\begin{enumerate}[label=\alph*)]
		\item $\gr M\coloneqq\bigoplus_{i\geq0}\mathfrak m_G^iM/\mathfrak m_G^{i+1}M$, which is a graded $\gr_\nu\mathbf F\llbracket G^{p^N}\rrbracket$-module via its $\gr\mathbf F\llbracket G\rrbracket$-module structure;
		\item $\grint M\coloneqq\bigoplus_{i\geq0}\mathfrak m_G^{ip^N}M/\mathfrak m_G^{(i+1)p^N}M$, naturally a graded $\gr_\nu\mathbf F\llbracket G^{p^N}\rrbracket$-module because $\mathfrak m_G^{ip^N}\cap\mathbf F\llbracket G^{p^N}\rrbracket=\mathbf F\llbracket G^{p^N}\rrbracket_{\nu\geq ip^N}$;
		\item $\gres M\coloneqq\bigoplus_{i\geq0}\mathfrak n_iM/\mathfrak n_{i+1}M$, where $\mathfrak n_i\coloneqq\mathbf F\llbracket G^{p^N}\rrbracket_{\nu\geq ip^N}$, naturally a graded $\gr_\nu\mathbf F\llbracket G^{p^N}\rrbracket$-module.
	\end{enumerate}
	We write ``$\mathfrak n_i$'' even though residue classes of its elements will be considered in degree $ip^N$.\\
	
	As already explained, the condition for $\pi$ to be in $\mathcal C$ is in terms of $\gr\pi^\vee$, which depends a priori not just on $\pi|_{G^{p^N}}$. We will show that the annihilation condition that $\gr\pi^\vee$ should satisfy to be in $\mathcal C$ can be equivalently phrased in terms of the $\gr_\nu\mathbf F\llbracket G^{p^N}\rrbracket$-module $\gres\pi^\vee$, which is a module clearly depending only on the restriction of $\pi$ to $G^{p^N}$. The graded module $\grint\pi^\vee$ will be used as an intermediary step in our proof of this equivalent formulation.
	\section{The annihilation condition}
	The classes of the $A_i-1$ and $B_i-1$ in $\gr\mathbf F\llbracket G\rrbracket$ form a basis of $\mathfrak m_G/\mathfrak m_G^2$ and they already generate $\gr\mathbf F\llbracket G\rrbracket$ as an $\mathbf F$-algebra. In fact, the $\mathbf F$-span of the classes of the $C_i-1$ in $\mathfrak m_G^2/\mathfrak m_G^3$ is precisely the image of the commutator bracket $[-,-]\colon\gr^1\mathbf F\llbracket G\rrbracket\times\gr^1\mathbf F\llbracket G\rrbracket\to\gr^2\mathbf F\llbracket G\rrbracket$. Let us denote these classes by $a_i,b_i,$ and $c_i$, respectively. These observations already yield the first part of following proposition:
	\begin{proposition}[{\cite[Thm.~5.3.4]{BHHMS1}},{\cite[Thm.~2.6]{hu2024modprepresentationsquaternion}}]
		The quotient $\gr\mathbf F\llbracket G\rrbracket/(c_0,\dots,c_{f-1})$ is isomorphic to the polynomial ring $\mathbf F[a_0,b_0,\dots,a_{f-1},b_{f-1}]$ and the sequence  $(c_0,\dots,c_{f-1})$ is a regular sequence of central elements in $\gr\mathbf F\llbracket G\rrbracket$.
	\end{proposition}
	Note that, since $N\geq1$, the ring $\gr_\nu\mathbf F\llbracket G^{p^N}\rrbracket$ has even better properties:
	\begin{proposition}\label{centrality}
		The ring $\gr_\nu\mathbf F\llbracket G^{p^N}\rrbracket=\mathbf F[a_0^{p^N},b_0^{p^N},c_0^{p^N},\dots,a_{f-1}^{p^N},b_{f-1}^{p^N},c_{f-1}^{p^N}]$ is a commutative polynomial $\mathbf F$-algebra. It is a subring of the center of $\gr\mathbf F\llbracket G\rrbracket$.
	\end{proposition}
	\begin{proof}
		The inclusion $\mathbf F\llbracket G^{p^N}\rrbracket\hookrightarrow\mathbf F\llbracket G\rrbracket$ induces an inclusion after taking $\gr_\nu$. We have $\gr_\nu\mathbf F\llbracket G\rrbracket=\gr\mathbf F\llbracket G\rrbracket$ by Proposition \ref{maxideals}.\\
		We first prove that the $a_i^{p^N}$ and $b_i^{p^N}$ are central. This follows from their centrality in $\gr\mathbf F\llbracket G\rrbracket$ which is seen as follows: take any generator $x\in\{a_0,b_0,\dots,a_{f-1},b_{f-1}\}$; then, as mentioned, $[a_i,x]$ lies in the span of the $c_j$; by the previous proposition $[a_i,x]$ is therefore central, from this and induction we get the formula $[a_i^\ell,x]=\ell a_i^{\ell-1}[a_i,x]$ for all $\ell\geq1$, since indeed we have \begin{equation*}
			[a_i^{\ell+1},x]=a_i[a_i^\ell,x]+a_ixa_i^\ell-xa_i^{\ell+1}=a_i[a_i^\ell,x]+[a_i,x]a_i^\ell
		\end{equation*} and to conclude from this last expression one uses the induction hypothesis and centrality of $[a_i,x]$; in particular, we see that $a_i^p$ is already central (the proof works the same for $b_i$ of course).\\
		Consider now the graded polynomial ring $\mathbf F[y_1,\dots,y_{3f}]$ which has the first $2f$ variables homogeneous in degree $1$ and the last $f$ variables homogeneous in degree $2$. It can be interpreted as the associated graded ring to $\mathbf F\llbracket p^N\mathbf Z_p^{3f}\rrbracket$ where $p^N\mathbf Z_p^{3f}$ is given the $p$-valuation induced from its homeomorphism to $G^{p^N}$. The homeomorphism yields an isomorphism of $\mathbf F$-modules $\mathbf F[y_1,\dots,y_{3f}]\xrightarrow{\sim}\gr_\nu\mathbf F\llbracket G^{p^N}\rrbracket$. By the previous proposition and the centrality of the $a_i^{p^N}$ and the $b_i^{p^N}$ this is actually a morphism of $\mathbf F$-algebras.
	\end{proof}
	For now and the rest of the text fix a (soon to be assumed homogeneous) two-sided ideal $$J\subset\gr\mathbf F\llbracket G\rrbracket$$ that yields a commutative ring when quotienting by it. That is, $J$ contains the $c_i$ (in fact, any left or right ideal containing all of the $c_i$ is two-sided). It is finitely generated, as left or right ideal, and we can therefore write it as \begin{equation*}\label{J}
		\begin{aligned}
			&J=\mathfrak f+\mathfrak c,\\
			&\mathfrak f\coloneqq(f_1,\dots,f_n),\quad\mathfrak c\coloneqq(c_0,\dots,c_{f-1}),\\
			&f_i=\sum_{\underline m,\underline n\in\mathbf Z_{\geq0}^f}\alpha_{\underline m,\underline n,i}a_0^{m_0}\cdots a_{f-1}^{m_{f-1}}b_0^{n_0}\cdots b_{f-1}^{n_{f-1}};
		\end{aligned}
	\end{equation*} we require our generators to generate the ideal both as left and as right ideal.
	\begin{definition}
		Let $\mathcal C_J$ be the category of $\mathbf F\llbracket G\rrbracket$-modules $M$ such that $\gr M$ is killed by a power of $J$ (with morphisms being the usual ones).
	\end{definition}
	If $\tilde J$ is the smallest homogeneous ideal containing $J$ then $\mathcal C_{\tilde J}=\mathcal C_J$, so without loss of generality we assume $J$ to be homogeneous. We also assume that our $f_i$ are all homogeneous.\\
	Let us now define an analogous ideal $J_N\subset\gr_\nu\mathbf F\llbracket G^{p^N}\rrbracket$:
	\begin{equation*}\label{JN}
		\begin{aligned}
			&J_N\coloneqq\tilde{\mathfrak f}+(c_0^{p^N},\dots,c_{f-1}^{p^N}),\\
			&\tilde{\mathfrak f}\coloneqq(\tilde f_1,\dots,\tilde f_n),\\
			&\tilde f_i\coloneqq\sum_{\underline m,\underline n\in\mathbf Z_{\geq0}^f}\alpha_{\underline m,\underline n,i}^{p^N}a_0^{p^Nm_0}\cdots a_{f-1}^{p^Nm_{f-1}}b_0^{p^Nn_0}\cdots b_{f-1}^{p^Nn_{f-1}}.
		\end{aligned}
	\end{equation*}
	From the homogeneity of $J$ we see that $J_N$ is also homogeneous.
	\section{Comparing annihilation conditions and proof of Theorem \ref{main}}
	\begin{lemma}
		Let $M$ be an $\mathbf F\llbracket G\rrbracket$-module. Then $M\in\mathcal C_J$ if and only if $\gr M$ is killed by a power of $J_N$.
	\end{lemma}
	\begin{proof}
		Suppose first that $M$ is in $\mathcal C_J$. Since $\gr\mathbf F\llbracket G\rrbracket/\mathfrak c$ is commutative, we have $f_i^{p^N}\equiv\tilde f_i\mod\mathfrak c$ for all $i$ so that $J_N\subset J$. Therefore $\gr M$ is killed by a power of $J_N$.\\
		Now suppose $\gr M$ killed by a power of $J_N$. Since $J$ has $f+n$ generators it follows from the pigeonhole principle (adding $\mathfrak c$ to avoid thinking about non-commutativity issues) that $$J^{(f+n)p^N}\subset (f_1^{p^N},\dots,f_n^{p^N},c_0^{p^N},\dots,c_{f-1}^{p^N})+\mathfrak c.$$ (One can interpret it as the left-ideal or the right-ideal generated by those elements and one obtains the same (two-sided) ideal in both cases.)
		The rightmost ideal is contained in $\mathfrak c+J_N\gr\mathbf F\llbracket G\rrbracket$ by the previously mentioned congruence and because $J_N\gr\mathbf F\llbracket G\rrbracket$ is a two-sided ideal (by Proposition \ref{centrality} it is generated by central elements) we get  $$\left(\mathfrak c+J_N\gr\mathbf F\llbracket G\rrbracket\right)^{fp^N}\subset\mathfrak c^{fp^N}+J_N\gr\mathbf F\llbracket G\rrbracket\subset J_N\gr\mathbf F\llbracket G\rrbracket,$$ from which it follows that $\gr M$ is killed by a power of $J$.
	\end{proof}
	\begin{lemma}\label{anotation}
		Let $M$ be an $\mathbf F\llbracket G\rrbracket$-module. Then $M\in\mathcal C_J$ if and only if $\grint M$ is killed by a power of $J_N$.
	\end{lemma}
	\begin{proof}
		Suppose first that $M$ is in $\mathcal C_J$, so that $\gr M$ is killed by a power of $J_N$ by the previous lemma, say by $J_N^\ell$, $\ell\geq1$. The ideal $J_N^\ell$ is a homogeneous ideal of $\gr_{\nu}\mathbf F\llbracket G^{p^N}\rrbracket$ and therefore generated by residue classes of elements $g_1,\dots g_t\in\mathbf F\llbracket G^{p^N}\rrbracket$ that are homogenous of degree $\ell_1p^N,\dots,\ell_tp^N$, respectively, with the $\ell_i\in\mathbf Z_{\geq0}$ (remember that the graded ring is $p^N\mathbf Z$-graded). From $J_N^\ell\gr M=0$ we find that $g_i\mathfrak m_G^kM\subset\mathfrak m_G^{k+\ell_ip^N+1}M$ for all $i$ and $k$. Iteration yields $$g_i^{p^N}\mathfrak m_G^kM\subset\mathfrak m_G^{k+\ell_ip^{2N}+p^N}M.$$
		Applied to all $k\in p^N\mathbf Z_{\geq0}$ we find that $J_N^{t\ell p^N}$ (even already $J_N^{\ell p^N}$) kills $\grint M$.\\
		Conversely, suppose $\grint M$ is killed by a power of $J_N$, say by the ideal $J_N^\ell$ for which we use the same description as above. Let $j\geq0$ and write $\frac j{p^N}=\lfloor\frac j{p^N}\rfloor+\{\frac j{p^N}\}$, where $\lfloor-\rfloor$ denotes the floor function. Then $$g_i\mathfrak m_G^jM\subset g_i\mathfrak m_G^{\lfloor\frac j{p^N}\rfloor p^N}M\subset\mathfrak m_G^{\left(\lfloor\frac j{p^N}\rfloor+\ell_i+1\right)p^N}M=\mathfrak m_G^{j+\ell_ip^N+\left(1-\{\frac j{p^N}\}\right)p^N}M\subset\mathfrak m_G^{j+\ell_ip^N+1}M,$$ which proves $J_N^\ell\gr M=0$, so that $M\in\mathcal C_J$ by the previous lemma.
	\end{proof}
	\begin{lemma}
		For every $k\in\mathbf Z_{\geq0}$ we have the following two inclusions $$\mathfrak n_k\mathbf F\llbracket G\rrbracket\subset\mathfrak m_G^{kp^N}\subset\mathfrak n_{k-4f}\mathbf F\llbracket G\rrbracket,$$
		where we note that the previous definition $\mathfrak n_i\coloneqq\mathbf F\llbracket G^{p^N}\rrbracket_{\nu\geq ip^N}$ still makes sense for $i<0$.
	\end{lemma}
	\begin{proof}
		The first inclusion is trivial. For the second, notice first of all that each element of $\mathfrak m_G^{kp^N}$ can be written as an infinite $\mathbf F$-linear combination of elements of the form $$x_{\underline k,\underline\ell,\underline m}\coloneqq(A_0-1)^{k_0}\cdots(A_{f-1}-1)^{k_{f-1}}(B_0-1)^{\ell_0}\cdots(B_{f-1}-1)^{\ell_{f-1}}(C_0-1)^{m_0}\cdots(C_{f-1}-1)^{m_{f-1}}$$ with $\nu(x_{\underline k,\underline\ell,\underline m})=\sum_ik_i+\ell_i+2m_i\geq kp^N$, by (\ref{complgrpring}) and Proposition \ref{maxideals}.\\
		Then set \begin{align*}\tau(x_{\underline k,\underline\ell,\underline m})\coloneqq&(A_0-1)^{\lfloor\frac{k_0}{p^N}\rfloor p^N}\cdots(A_{f-1}-1)^{\lfloor\frac{k_{f-1}}{p^N}\rfloor p^N}(B_0-1)^{\lfloor\frac{\ell_0}{p^N}\rfloor p^N}\cdots(B_{f-1}-1)^{\lfloor\frac{\ell_{f-1}}{p^N}\rfloor p^N}\\\cdot&(C_0-1)^{\lfloor\frac{m_0}{p^N}\rfloor p^N}\cdots(C_{f-1}-1)^{\lfloor\frac{m_{f-1}}{p^N}\rfloor p^N}(A_0-1)^{\{\frac{k_0}{p^N}\} p^N}\cdots(A_{f-1}-1)^{\{\frac{k_{f-1}}{p^N}\} p^N}\\\cdot&(B_0-1)^{\{\frac{\ell_0}{p^N}\} p^N}\cdots(B_{f-1}-1)^{\{\frac{\ell_{f-1}}{p^N}\} p^N}(C_0-1)^{\{\frac{m_0}{p^N}\} p^N}\cdots(C_{f-1}-1)^{\{\frac{m_{f-1}}{p^N}\} p^N},\end{align*} where we again use the notation $\lfloor-\rfloor$ and $\{-\}$ for the floor function and the remaining fractional part, respectively.\\ Note that $\tau(x_{\underline k,\underline\ell,\underline m})\in\mathfrak n_{k-4f}\mathbf F\llbracket G\rrbracket$ because \begin{align*}
			\nu\Big((A_0-1)^{\lfloor\frac{k_0}{p^N}\rfloor p^N}\cdots(A_{f-1}-1)^{\lfloor\frac{k_{f-1}}{p^N}\rfloor p^N}&(B_0-1)^{\lfloor\frac{\ell_0}{p^N}\rfloor p^N}\cdots(B_{f-1}-1)^{\lfloor\frac{\ell_{f-1}}{p^N}\rfloor p^N}\\\cdot&(C_0-1)^{\lfloor\frac{m_0}{p^N}\rfloor p^N}\cdots(C_{f-1}-1)^{\lfloor\frac{m_{f-1}}{p^N}\rfloor p^N}\Big)\\=&\sum_i\lfloor\frac{k_i}{p^N}\rfloor p^N+\lfloor\frac{\ell_i}{p^N}\rfloor p^N+2\lfloor\frac{m_i}{p^N}\rfloor p^N\\\geq&(k-4f)p^N.
		\end{align*} Further note that $\nu(\tau(x_{\underline k,\underline\ell,\underline m}))=\nu(x_{\underline k,\underline\ell,\underline m})$ and that $\tau(x_{\underline k,\underline\ell,\underline m})-x_{\underline k,\underline\ell,\underline m}\in\mathfrak m_G^{\nu(x_{\underline k,\underline\ell,\underline m})+1}$ by the centrality of the $a_i^{p^N},b_i^{p^N},c_i^{p^N}$ in $\gr\mathbf F\llbracket G\rrbracket$. Consequently, $x_{\underline k,\underline\ell,\underline m}$ can be written as $\tau(x_{\underline k,\underline\ell,\underline m})$ plus an (infinite(?)) $\mathbf F$-linear combination of $x_{\underline k',\underline\ell',\underline m'}$ with $\nu(x_{\underline k',\underline\ell',\underline m'})>\nu(x_{\underline k,\underline\ell,\underline m})$. These $x_{\underline k',\underline\ell',\underline m'}$ can in turn be written as $\tau(x_{\underline k',\underline\ell',\underline m'})$ (in $\mathfrak n_{k-4f}\mathbf F\llbracket G\rrbracket$) plus $x_{\underline k',\underline\ell',\underline m'}-\tau(x_{\underline k',\underline\ell',\underline m'})$ (these being again of higher $\nu$-valuation), and so on.\\
		Repeating this process ad infinitum, we find that $x_{\underline k,\underline\ell,\underline m}$ is in the closure of $\mathfrak n_{k-4f}\mathbf F\llbracket G\rrbracket$, but since this is a finitely generated right-$\mathbf F\llbracket G\rrbracket$-module it is already closed. This proves the second inclusion, once more using the closedness of $\mathfrak n_{k-4f}\mathbf F\llbracket G\rrbracket$.
	\end{proof}
	\begin{lemma}\label{lastlemma}
		Let $M$ be an $\mathbf F\llbracket G\rrbracket$-module. Then $M\in\mathcal C_J$ if and only if $\gres M$ is killed by a power of $J_N$.
	\end{lemma}
	\begin{proof}
		Suppose first that $J_N^\ell$ kills $\grint M$ (\emph{i.e.~}$M\in\mathcal C_J$ by Lemma \ref{anotation}), using again the same notation as in the proof of Lemma \ref{anotation}. This means that, for any $k$, $g_i\mathfrak m_G^{kp^N}M\subset\mathfrak m_G^{(k+\ell_i+1)p^N}M$. Therefore, for example, $$g_i^{4f+1}\mathfrak m_G^{kp^N}M\subset\mathfrak m_G^{(k+(4f+1)\ell_i)p^N+(4f+1)p^N}M$$ and thus $$g_i^{4f+1}\mathfrak n_kM\subset\mathfrak n_{k+(4f+1)\ell_i+1}M$$ by the previous lemma.
		However, the class of $g_i^{4f+1}$ in $\gr_\nu\mathbf F\llbracket G^{p^N}\rrbracket$ is of degree $(4f+1)\ell_ip^N$ from which it follows that the class $g_i^{4f+1}$ kills $\gres M$. Generally one has $J_N^{(4f+1)\ell}\gres M=0$ because the class of any product $g_{i_1}\cdot\cdots\cdot g_{i_{4f+1}}$ is in degree $dp^N$, with $d\coloneqq\sum_{j=1}^{4f+1}\ell_{i_j}$, (and these products generate $J_N^{(4f+1)\ell}$) while for any $k$ $$g_{i_1}\cdot\cdots\cdot g_{i_{4f+1}}\mathfrak n_kM\subset\mathfrak n_{k+d+1}M.$$
		In the other direction one similarly proves, using the previous lemma again, that if $J_N^\ell\gres M=0$, then $J_N^{(4f+1)\ell}\grint M=0$.
	\end{proof}
	\begin{proof}[Proof of Theorem \ref{main}]
		Assume $N$ is big enough so that, modulo $Z_1$, $H$ contains $G^{p^N}$. Take for $J$ the ideal $J_{\mathcal C}$ of the definition of $\mathcal C$. Then $\gres\pi'^\vee$ is killed by a power of $J_N$ by Lemma \ref{lastlemma}. By the isomorphism of $\pi$ and $\pi'$ when restricted to $H$, we find that their duals are isomorphic $\mathbf F\llbracket G^{p^N}\rrbracket$-modules. Therefore $\gres\pi^\vee$ is killed by the same power of $J_N$, so that $\pi^\vee\in\mathcal C_J$ by Lemma \ref{lastlemma} again and hence $\pi$ is in $\mathcal C$ as well (note that $\pi$ is admissible because $\pi'$ is).
	\end{proof}
	\bibliographystyle{alpha}
	\bibliography{CatC} 
\end{document}